\documentclass[reqno]{amsart}

\usepackage{amssymb, amsfonts}

\usepackage{fullpage}

\usepackage[hidelinks]{hyperref}
\usepackage{url}
\usepackage{cleveref}

\usepackage{tikz-cd}

\usepackage{mathtools}

\newtheorem{theorem}[equation]{Theorem}
\newtheorem{lemma}[equation]{Lemma}
\newtheorem{proposition}[equation]{Proposition}
\newtheorem{corollary}[equation]{Corollary}
\newtheorem*{conjecture}{Conjecture}
\newtheorem*{maintheorem}{Main Theorem}

\theoremstyle{definition}
\newtheorem{definition}[equation]{Definition}

\theoremstyle{remark}
\newtheorem{remark}[equation]{Remark}

\numberwithin{equation}{section}

\DeclareMathOperator{\ch}{ch}

\DeclareMathOperator{\even}{even}
\DeclareMathOperator{\odd}{odd}

\newcommand{\ZZ}{\mathbb{Z}}
\newcommand{\QQ}{\mathbb{Q}}

\newcommand{\CC}{\mathbb{C}}

\newcommand{\gl}{\mathfrak{gl}}
\newcommand{\so}{\mathfrak{so}}

\begin{document}
	\title[Bound on second Betti of HK]{A conjectural bound on the second Betti number for hyper-K\"ahler manifolds}
	
	\author[Y.-J. Kim]{Yoon-Joo Kim}
	\address{Stony Brook University, Department of Mathematics, Stony Brook, NY 11794-3651}
	\email{yoon-joo.kim@stonybrook.edu}
	
	\author[R. Laza]{Radu Laza}
	\address{Stony Brook University, Department of Mathematics, Stony Brook, NY 11794-3651}
	\email{radu.laza@stonybrook.edu}
	
	\thanks{The two authors were partially supported by NSF grant DMS-1802128.}
	\maketitle
	
	\begin{abstract}
		In previous work (\cite{gklr19}), we noted that the known cases of hyper-K\"ahler manifolds satisfy a natural condition on the LLV decomposition of the cohomology; informally, the Verbitsky component is the dominant representation in the LLV decomposition.  Assuming this condition holds for all hyper-K\"ahler manifolds, we obtain an upper bound for the second Betti number  in terms of the dimension.  
	\end{abstract}

\section{Introduction}
A fundamental open question in the theory of compact hyper-K\"ahler manifold is the boundedness question: {\it are there finitely many diffeomorphism types of hyper-K\"ahlers in a given dimension?} In accordance with the Torelli principle, Huybrechts \cite[Thm 4.3]{huy03} proved that there are finitely many diffeomorphism types of hyper-K\"ahler manifolds once the dimension and the (unnormalized) Beauville--Bogomolov lattice $(H^2(X, \ZZ), q_X)$ are fixed. Thus, bounding the hyper-K\"ahler manifolds is equivalent to bounding the second Betti number $b_2 = b_2(X)$, and then the Beauville--Bogomolov form (e.g., the discriminant).  In dimension $2$, a compact hyper-K\"ahler manifold is always a K3 surface, thus $b_2 = 22$. In dimension $4$, Beauville and Guan \cite{guan01} gave a sharp bound $b_2\le 23$ (in fact, Guan showed $3 \le b_2 \le 8$ or $b_2 = 23$). 
For some further partial results on bounding $b_2$ see Remark \ref{rem_saw}. The purpose of this note is to give a conjectural bound on  $b_2(X)$ for an arbitrary compact hyper-K\"ahler manifold $X$ of dimension $2n$. Our bound depends on a natural conjectural condition  satisfied by the Looijenga--Lunts--Verbitsky (LLV) decomposition of the cohomology $H^*(X)$ for hyper-K\"ahler manifolds $X$.

To state our results, let us recall that Verbitsky \cite{ver95} and Looijenga--Lunts \cite{ll97} noted that the cohomology $H^*(X)$ of a hyper-K\"ahler manifold admits a natural action by the Lie algebra $\mathfrak g = \mathfrak{so}(b_2 + 2)$, generalizing the usual hard Lefschetz theorem. As a $\mathfrak g$-module, the cohomology of a hyper-K\"ahler manifold $X$ decomposes as 
\begin{equation} \label{eq:llv_decomp}
	H^* (X) = \sideset{}{_{\mu}} {\bigoplus} V_{\mu}^{\oplus m_{\mu}},
\end{equation}
where $V_{\mu}$ indicates an irreducible $\mathfrak g$-module of highest weight $\mu=(\mu_0, \cdots, \mu_r)$, with  $r= \left \lfloor \frac{b_2(X)}{2} \right \rfloor=\mathrm{rk}\ \mathfrak g-1$. We refer to $\mathfrak g$ as the {\it LLV algebra} of $X$, and to \eqref{eq:llv_decomp} as the {\it LLV decomposition} of $H^*(X)$ (see \cite{gklr19} for further discussion). Motivated by the behavior of the LLV decomposition in the known cases of hyper-K\"ahler manifolds \cite{gklr19}, we have made the following conjecture.

\begin{conjecture}[{\cite{gklr19}}]
	Let $X$ be a compact hyper-K\"ahler manifold of dimension $2n$. Then the weights  $\mu = (\mu_0, \cdots, \mu_r)$ occurring  in the LLV decomposition \eqref{eq:llv_decomp} of $H^*(X)$ satisfy
	\begin{equation} \label{eq:conj}
		\mu_0 + \cdots + \mu_{r-1} + |\mu_r| \le n. 
	\end{equation}
\end{conjecture}

The conjecture holds for all currently known examples of compact hyper-K\"ahler manifolds (cf. \cite[\S 1]{gklr19}). Furthermore, the equality in \eqref{eq:conj} holds for {\it the Verbitsky component}, an irreducible $\mathfrak g$-submodule with highest weight $\mu = (n,0,\cdots,0)$ that is always present in $H^*(X)$. This shows \eqref{eq:conj} is sharp. Beyond the evidence given by the validity of \eqref{eq:conj} in the known cases, we have some partial arguments of motivic nature (and depending on standard conjectures) showing that at least \eqref{eq:conj} is plausible. This will be discussed elsewhere.

The purpose of this note is to show that conjecture \eqref{eq:conj} implies a general bound on $b_2(X)$.

\begin{maintheorem}
	Let $X$ be a compact hyper-K\"ahler manifold of dimension $2n$. If the condition \eqref{eq:conj} holds for $X$, then
	\begin{equation} \label{bound_b2}
		b_2(X) \le \begin{cases}
		\frac{21 + \sqrt{96n + 433}}{2} & \mbox{if } H^*_{\odd} (X) = 0 \\
		2k + 1 & \mbox{if } H^k (X) \neq 0 \mbox{ for some odd } k
		\end{cases} .
	\end{equation}
\end{maintheorem}

\begin{remark}
	A slightly weaker version of \eqref{bound_b2} is 
	\begin{equation*}
		b_2(X) \le \max \left\{ \tfrac{21 + \sqrt{96n + 433}}{2}, 4n-1 \right\},
	\end{equation*}
	which reads explicitly 
	\begin{table}[h]
	\renewcommand{\arraystretch}{1.3}
		\begin{tabular}{|c||c|c|c|c|c|c|c|c|}
			\hline 
			$n$ & $1$ & $2$ & $3$ & $4$ & $5$ & $6$ & $7$ & $\ge 8$ \\ 
			\hline 
			$b_2(X) \le$ & $22$ & $23$ & $23$ & $24$ & $25$ & $26$ & $27$ & $4n-1$ \\ 
			\hline 
		\end{tabular}\ \ .
	\end{table}
	
	\noindent In low dimensions, our bounds agree with the known results and seem fairly sharp. For instance, we know $\mathrm{K3}^{[n]}$ type hyper-K\"ahler manifolds have $b_2 = 23$ and $\mathrm{OG10}$ manifolds have $b_2 = 24$. These examples almost reach our maximum bound of $b_2$ for low dimensions. Similarly, $\mathrm{Kum}_n$ type hyper-K\"ahler manifolds have $H^3 (X) \neq 0$ and $b_2 = 7$, showing also that the second inequality in \eqref{bound_b2} is sharp.
\end{remark}

\begin{remark} \label{rem_saw}
	Sawon \cite{saw15} and Kurnosov \cite{kur15} have previously obtained the same bounds for $3 \le n \le 5$, and also predicted the general formula \eqref{bound_b2} when $H^*_{\odd} (X) = 0$. However, their results were based on the assumption that an irreducible module $V_{\mu}$ is determined by the shape of its Hodge diamond. In general, the shape of the Hodge diamond of $V_{\mu}$ is controlled only by  the first two coefficients $\mu_0, \mu_1$ (see \cite[\S2.2]{gklr19}).
\end{remark}

A few words about the proof of our conjectural bound. First, in \cite[\S 1]{gklr19}, we have already obtained that the condition \eqref{eq:conj} has consequences on the odd cohomology (specifically, if $b_2 \ge 4n$ then there should be no odd cohomology). A slight generalization of the argument in loc. cit. then gives the second inequality in \eqref{bound_b2}. The main content of this note is the control of the even cohomology under the assumption \eqref{eq:conj}. Essentially, our argument is a representation theoretic refinement of Beauville's argument that $b_2 \le 23$ for hyper-K\"ahler fourfolds. Namely, the starting point is Salamon's relation  \cite{sal96}, a linear relation satisfied by the Betti numbers of hyper-K\"ahler manifolds. Inspired by the shape of it, we define a numerical function $s(W)$ for a $\mathfrak g$-module $W$ and verify its basic properties, most importantly $s ( W_1 \oplus W_2 ) \le \max \{ s(W_1), s(W_2) \}$. In this setting, Salamon's relation reads $s(H^*(X)) = \frac{n}{3}$. Now the punchline is an explicit formula for $s(V_{\mu})$ for irreducible $\mathfrak g$-modules $V_{\mu}$ (\Cref{thm:generalized_weyl}), which is obtained by applying the Weyl character formula. Combining it with \eqref{eq:conj}, we conclude
\[ \frac{n}{3} = s(H^*(X)) \le s(V_{(n,0,\cdots,0)}) = \frac{8n (b_2+n)}{(b_2+1) (b_2+2)} ,\]
which in turn gives the first inequality in \eqref{bound_b2}.

\section{Cohomology of compact hyper-K\"ahler manifolds}
\label{sec:backgrounds}
We briefly review some relevant results on the cohomology of hyper-K\"ahler manifolds.  Let $X$ be a compact hyper-K\"ahler manifold of dimension $2n$ and $H^* (X) = H^* (X, \CC)$. Let $\mathfrak g \subset \gl (H^* (X))$ be the Lie algebra generated by all the Lefschetz and dual Lefschetz operators associated to elements in $H^2 (X)$ (cf. \cite{ver95}, \cite{ll97}). We call this the Looijenga--Lunts--Verbitsky (LLV) algebra of $X$. Let
\[ (V, q) = (H^2 (X), q_{X}) \oplus U \]
be the Mukai completion of $H^2 (X)$ equipped with the Beauville--Bogomolov form, and set $r = \left \lfloor \frac{b_2(X)}{2} \right \rfloor$. Then $\mathfrak g$ is isomorphic to the special orthogonal Lie algebra $\so (V, q) \cong \so (b_2 + 2, \CC)$ of rank $r+1$. The cohomology $H^*(X)$ of a hyper-K\"ahler manifold $X$ admits a $\mathfrak g$-module structure, generalizing the hard Lefschetz theorem. We refer to the $\mathfrak g$-module irreducible decomposition \eqref{eq:llv_decomp} of $H^* (X)$ as the {\it LLV decomposition} of the cohomology (see \cite[\S3]{gklr19} for some examples).

Fix a Cartan and a Borel subalgebra of $\mathfrak g$. Representation theory of $\so (V, q)$ depends on the parity of $\dim V = b_2 + 2$. If $b_2 = 2r$ is even, then we can fix a suitable basis $\varepsilon_0, \cdots, \varepsilon_r$ of the dual Cartan subalgebra such that the $2r+2$ associated weights of the standard module $V$ are $\pm \varepsilon_0, \cdots, \pm \varepsilon_r$. Similarly, if $b_2 = 2r+1$ is odd, then we can choose $\varepsilon_i$ such that $V$ has the $2r+3$ associated weights $0, \pm \varepsilon_0, \cdots, \pm \varepsilon_r$. (Note that the index of the basis starts from $0$.) Any dominant integral weight $\mu$ can be expressed in this basis as
\[ \mu = (\mu_0, \cdots, \mu_r) = \sum_{i=0}^r \mu_i \varepsilon_i .\]
Here if $b_2 = 2r$ is even, then $\mu_i$ satisfy the condition $\mu_0 \ge \cdots \ge \mu_{r-1} \ge |\mu_r| \ge 0$ and $\mu_i$ are either all integers or all half-integers. If $b_2 = 2r+1$ is odd, then $\mu_0 \ge \cdots \ge \mu_r \ge 0$ and $\mu_i$ are again either all integers or all half-integers. It will be important whether $\mu_i$ are integers or half-integers, so we define:
\begin{definition}
	Let $\mu = (\mu_0, \cdots, \mu_r)$ be a dominant integral weight of $\mathfrak g = \so (V, q)$.
	\begin{enumerate}
		\item If all $\mu_i$ are integers, we say $\mu$ is \emph{even}. If all $\mu_i$ are half-integers (i.e., $\mu_i \in \frac{1}{2} \ZZ \setminus \ZZ$), we say $\mu$ is \emph{odd}.
		\item An irreducible $\mathfrak g$-module $V_{\mu}$ of highest weight $\mu$ is called even (resp. odd) if $\mu$ is even (resp. odd).
		\item A $\mathfrak g$-module $W$ is called even (resp. odd) if all of its irreducible components $V_{\mu}$ are even (resp. odd).
	\end{enumerate}
\end{definition}

\noindent By \cite[Prop 2.35]{gklr19}, the even (odd) cohomology $H^*_{\even} (X)$ is always an even (resp. odd) $\mathfrak g$-module.

When $\mu$ is even and has multiple $0$'s at the end of its coordinate expression $(\mu_0, \cdots, \mu_r)$, we will simply omit the last $0$'s. For example, the notation $(m) = (m, 0, \cdots, 0)$ refers to the integral weight $m \varepsilon_0$. In geometric situation for hyper-K\"ahler manifolds, the subalgebra of $H^* (X)$ generated by $H^2 (X)$ becomes an irreducible $\mathfrak g$-submodule of $H^* (X)$, which we call {\it the Verbitsky component} of $H^* (X)$. As a $\mathfrak g$-module, Verbitsky component is isomorphic to $V_{(n)}$ and it always occurs with multiplicity $1$ in the LLV decomposition.

Let $h$ be the degree operator on $H^* (X)$, the operator acting as multiplication by $k$ on $H^{2n+k} (X)$. For a suitable choice of a Cartan and a Borel subalgebra, we can assume $h = \varepsilon_0^{\vee}$ (e.g., \cite[(2.28)]{gklr19}). By definition, the degree decomposition of the cohomology
\[ H^* (X) = \bigoplus_{k=-2n}^{2n} H^{2n + k} (X) \]
is the $h$-eigenspace decomposition. In general, an arbitrary $\mathfrak g$-module admits the $h$-eigenspace decomposition
\begin{equation} \label{eq:h_eigendecomp}
	W = \bigoplus_{k \in \ZZ} W_k ,
\end{equation}
where $W_k$ denotes the eigenspace of $W$ with eigenvalue $k$. The eigenvalues $k$ are always integers by the following reason. Let $W(\theta)$ be the weight subspace of $W$ associated to a weight $\theta = \theta_0 \varepsilon_0 + \cdots + \theta_r \varepsilon_r$. Then $h = \varepsilon_0^{\vee}$ acts on $W(\theta)$ by $\langle \varepsilon_0^{\vee}, \theta \rangle = 2 \theta_0$, which is an integer since $\theta_0 \in \frac{1}{2} \ZZ$ for any weight $\theta$.

Consider the LLV decomposition of the cohomology
\begin{equation}
	H^* (X) = \sideset{}{_{\mu}} {\bigoplus}  V_{\mu}^{\oplus m_{\mu}} . \tag{\ref{eq:llv_decomp} (restated)}
\end{equation}
If $V_{\mu}$ is contained in the odd cohomology, then $\mu$ is odd by the above discussion. Hence all $\mu_i$ are half-integers, and in particular we have $\mu_i \ge \frac{1}{2}$ (possibly except for the last $|\mu_r| \ge \frac{1}{2}$, if $b_2$ is even). If we specifically assume $H^k (X) \neq 0$ for odd $k < 2n$, then there exists at least one irreducible component $V_{\mu}$ with $(V_{\mu})_{k-2n} \neq 0$. This means $h = \varepsilon_0^{\vee}$ acts on some part of $V_{\mu}$ by $k-2n$, so $V_{\mu}$ has an associated weight $\theta = \theta_0 \varepsilon_0 + \cdots + \theta_r \varepsilon_r$ with $\theta_0 = \frac{k}{2} - n$. This forces $\mu_0 \ge n - \frac{k}{2}$. Summarizing, we have
\[ \mu_0 \ge n - \frac{k}{2}, \qquad \mu_1, \cdots, \mu_{r-1}, |\mu_r| \ge \frac{1}{2} ,\]
which gives the following.

\begin{corollary} \label{cor:odd_weight_bound}
	Let $X$ be a compact hyper-K\"ahler manifold of dimension $2n$. Assume $H^k (X) \neq 0$ for some odd integer $k < 2n$. Then there exists a weight $\mu$ in \eqref{eq:llv_decomp} with $\mu_0 + \cdots + \mu_{r-1} + |\mu_r| \ge n - \frac{k}{2} + \frac{r}{2}$. \qed
\end{corollary}

Finally, let us recall Salamon's relation. Let $b_k = b_k(X)$ be the $k$-th Betti number of $X$. Salamon \cite{sal96} proved that the Betti numbers of hyper-K\"ahler manifolds $X$ satisfy a linear relation:
\[ \sum_{k=1}^{2n} (-1)^k (6k^2 - 2n) b_{2n + k} = n b_{2n} .\]
One can manipulate the identity into the following form
\begin{equation} \label{eq:salamon_old}
	\sum_{k=-2n}^{2n} (-1)^k k^2 b_{2n + k} = \frac{n}{3} e(X), 
\end{equation}
where $e(X) = \sum_{k=-2n}^{2n} (-1)^k b_{2n+k}$ is the topological Euler characteristic of $X$.

\section{Proof of Main Theorem}
Inspired by Salamon's relation \eqref{eq:salamon_old}, we define a constant $s(W)$ associated to an arbitrary $\mathfrak g$-module $W$.

\begin{definition}
	Let $W$ be a $\mathfrak g$-module and $W = \bigoplus_k W_k$ its $h$-eigenspace decomposition in \eqref{eq:h_eigendecomp}. Assume $\sum_k (-1)^k \dim W_k\neq 0$ (N.B. This is automatic if $W$ is either even or odd). Then we define a constant $s(W)\in \QQ$ associated to $W$ by
	\[ s(W) = \frac{\sum_{k \in \ZZ} (-1)^k k^2 \dim W_k}{\sum_{k \in \ZZ} (-1)^k \dim W_k} .\]
	In particular, if $e(X) \neq 0$, Salamon's relation \eqref{eq:salamon_old} reads
	\begin{equation} \label{eq:salamon_new}
		s(H^*(X)) = \frac{n}{3} .
	\end{equation}
\end{definition}

The case of odd cohomology will be easily handled by Corollary \ref{cor:odd_weight_bound}. Thus, we can focus on the case of vanishing odd cohomology (in particular, $e(X) \neq 0$). The main content then is to bound the value $s(H^* (X))$ in terms of $b_2$ and the LLV decomposition \eqref{eq:llv_decomp}. Once this is done, assuming our conjecture \eqref{eq:conj}, Salamon's relation \eqref{eq:salamon_new} leads to the desired inequality \eqref{bound_b2} between $b_2$ and $n$. Let us start from some straightforward properties of the constant $s(W)$.

\begin{proposition} \label{prop:s_prop1}
	Let $\{ W_i \}_{i \in I}$ be a finite set of $\mathfrak g$-modules with well-defined $s(W_i)$.
	\begin{enumerate}
		\item If all $W_i$ are simultaneously even or odd, then $\min_i \{ s(W_i) \} \le s \left( \bigoplus_i W_i \right) \le \max_i \{ s(W_i) \}$.
		\item $s \left( \bigotimes_i W_i \right) = \sum_i s(W_i)$.
	\end{enumerate}
\end{proposition}
\begin{proof}
	It is enough to prove the proposition for two $\mathfrak g$-modules $W$ and $W'$. Assume without loss of generality $s(W) \le s(W')$, and let us consider the case when $W$ and $W'$ are even (the odd case is similar). In this case, all eigenvalues $k$ of $W$ are even, so we have
	\[ \sum_k k^2 \dim W_k = s(W) \dim W, \qquad \sum_k k^2 \dim W'_k = s(W') \dim W' .\]
	Adding the two equalities and using $s(W) \le s(W')$ gives us the first item. 
	
	For the second item, we compute
	\begin{align*}
		\sum_k (-1)^k k^2 \dim (W \otimes W')_k &= \sum_k (-1)^k k^2 \left( \sum_{i+j = k} \dim W_i \dim W'_j \right) \\
		&= \sum_{i, j} (-1)^{i+j} (i^2 + 2ij + j^2) \dim W_i \dim W'_j \\
		&= \left( \sum_i (-1)^i i^2 \dim W_i \right) e(W') + \left( \sum_j (-1)^j j^2 \dim W'_j \right) e(W) \\
		& \qquad + 2 \left( \sum_i (-1)^i i \dim W_i \right) \left( \sum_j (-1)^j j \dim W'_j \right) .
	\end{align*}
	Here we used the notation $e(W) = \sum_i (-1)^i \dim W_i$ and $e(W') = \sum_j (-1)^j \dim W'_j$ for simplicity. Notice that $\sum_i (-1)^i i \dim W_i = 0$, since by Weyl symmetry we always have $\dim W_i = \dim W_{-i}$. This proves $\sum_k (-1)^k k^2 \dim (W \otimes W')_k = (\sum_i (-1)^i i^2 \dim W_i) e(W') + (\sum_j (-1)^j j^2 \dim W'_j) e(W)$. Dividing both hand sides by $e(W \otimes W') = e(W) e(W')$ gives us the result.
\end{proof}

\begin{remark}
	In fact, we can associate to an arbitrary $\mathfrak g$-module $W$ the following formal power series
	\[ S(W) = \sum_k (-1)^k \dim W_k \cdot \exp (kt) \ \ \in \ \ \QQ[[t]] .\]
	One can easily show 
	\[ S(W \oplus W') = S(W) + S(W') , \quad S(W \otimes W') = S(W)\cdot S(W') ,\]
	so that $S$ defines a \emph{ring homomorphism} from the representation ring $K(\mathfrak g)$ of $\mathfrak g$
	\[ S : K(\mathfrak g) \to \QQ [[t]] .\]
	By Weyl symmetry, we have $\dim W_k = \dim W_{-k}$, giving that all the odd degree terms of $S(W)$ vanish. Thus, we can write
	\[ S(W) = s_0 + s_2 t^2 + s_4 t^4 + \cdots \ \ \in \ \ \QQ[[t]] ,\qquad s_i = \frac{1}{i!} \sum_k (-1)^k k^i \dim W_k .\]
	From this perspective, our constant $s(W)$ is the ratio between the first two coefficients
	\[ s(W)=\frac{2s_2}{s_0} \]
	of the formal power series $S(W)$.
\end{remark}

A more interesting result is the explicit computation of $s(W)$ for irreducible $\mathfrak g$-modules $W = V_{\mu}$. Recall that the Lie algebra $\mathfrak g$ was isomorphic to $\so (b_2 + 2, \CC)$ and $r = \lfloor \frac{b_2}{2} \rfloor$, so that $\mathfrak g$ has rank $r+1$ and a dominant integral weight $\mu$ can be written as a tuple $(\mu_0, \cdots, \mu_r)$.

\begin{theorem} \label{thm:generalized_weyl}
With notations as above,	let $V_{\mu}$ be an irreducible $\mathfrak g$-module of highest weight $\mu$. If $\mu_r \ge 0$, then
	\[ s(V_{\mu}) = 8 \cdot \frac{\left( \sum_{i=0}^r \mu_i \right) b_2 + \left( \sum_{i=0}^r (\mu_i - i)^2 - i^2 \right)}{(b_2+1)(b_2+2)} .\]
	If $b_2$ is even and $\mu_r < 0$, then $s(V_{\mu}) = s(V_{\mu'})$ where $\mu' = (\mu_0, \cdots, \mu_{r-1}, -\mu_r)$.
\end{theorem}

We postpone the proof of \Cref{thm:generalized_weyl} to the following section. For now let us conclude the proof of our Main Theorem using this result. First, we note the following consequence of \Cref{thm:generalized_weyl}. Recall from \Cref{sec:backgrounds} that the notation $(m) = (m,0,\cdots,0)$ refers to the integral weight $m \varepsilon_0$.

\begin{corollary} \label{cor:s_prop2}
	\begin{enumerate}
		\item $s(V_{(m)}) = \frac{8m (b_2+m)} {(b_2+1)(b_2+2)}$ for $m \in \ZZ_{\ge 0}$.
		\item If $\mu$ is even, then $s(V_{\mu}) \le s(V_{(m)})$ for $m = \mu_0 + \cdots + \mu_{r-1} + |\mu_r|$.
		\item $s(V_{(m)}) \le s(V_{(n)})$ for $m \le n$.
	\end{enumerate}
\end{corollary}
\begin{proof}
	The first item is immediate from letting $\mu = m \varepsilon_0$ in \Cref{thm:generalized_weyl}. The third item follows from it directly. For the second item, using $s(V_{\mu}) = s(V_{\mu'})$ in \Cref{thm:generalized_weyl}, we may assume $\mu_r \ge 0$. Let us temporarily define a function $A(\mu)$ of a dominant integral weight $\mu$ by
	\[ A(\mu) = \sum_{i = 0}^r (\mu_i - i)^2 .\]
	Again using \Cref{thm:generalized_weyl}, one finds that the second item is equivalent to the inequality $A(\mu) \le A(m \varepsilon_0)$. For it, one first proves an inequality
	\begin{equation} \label{eq:induction}
		(\mu_i - i)^2 + (\mu_j - j)^2 < (\mu_i + 1 - i)^2 + (\mu_j - 1 - j)^2 \qquad \mbox{for } \ 0 \le i < j \le r ,
	\end{equation}
	which easily follows from $\mu_i \ge \mu_j$. The desired $A(\mu) \le A(m \varepsilon_0)$ follows from inductively applying the inequality \eqref{eq:induction} to modify the dominant integral weight $\mu$ until it reaches $m \varepsilon_0$.
\end{proof}

\begin{proof} [Proof of Main Theorem]
	Assume $H^k (X) \neq 0$ for some odd integer $k$. By \Cref{cor:odd_weight_bound}, there exists at least one component $V_{\mu} \subset H^*_{\odd} (X)$ with $\mu_0 + \cdots + \mu_{r-1} + |\mu_r| \ge n - \frac{k}{2} + \frac{r}{2}$. Thus, under the condition \eqref{eq:conj}, we get $r \le k$ and hence $b_2 = (2r \mbox{ or } 2r+1) \le 2k+1$.
	
	Now assume $H^*_{\odd} (X) = 0$. 	Among the irreducible components $V_{\mu}$ of the LLV decomposition \eqref{eq:llv_decomp}, we always have the Verbitsky component, which as a $\mathfrak g$-module is isomorphic to $V_{(n)}$. Thus, if we assume the condition \eqref{eq:conj} holds for $X$, then combining \Cref{cor:s_prop2} with \Cref{prop:s_prop1} gives us
	\[ s(H^* (X)) \le \max \{ s(V_{\mu}) : \mu \mbox{ appearing in \eqref{eq:llv_decomp}} \} = s(V_{(n)}) = \frac{8n (b_2 + n)} {(b_2 + 1)(b_2 + 2)} .\]
	On the other hand, we have Salamon's relation $s(H^* (X)) = \frac{n}{3}$ in \eqref{eq:salamon_new}. We conclude 
	\[ s(H^*(X)) = \frac{n}{3} \le \frac{8n (b_2 + n)} {(b_2 + 1)(b_2 + 2)} ,\]
	giving the desired bound on $b_2$ in Main Theorem.
\end{proof}

\section{Computation of $s(W)$ for irreducible $\mathfrak g$-modules}
In this section, we prove \Cref{thm:generalized_weyl} by using standard representation theoretic methods.

Let us first fix the notation. Here, we simply write $b = b_2(X)$. Let $(V, q)$ be a quadratic space of dimension $b+2$ and $\mathfrak g = \so (V, q)$ be the associated simple Lie algebra of type $\mathrm B_{r+1}$ / $\mathrm D_{r+1}$. We fix a Cartan and a Borel subalgebra of $\mathfrak g$ so that the positive and simple roots are well defined. We also use the following notation:
\begin{itemize}
	\item $\mathfrak W$ is the Weyl group of $\mathfrak g$;
	\item $R_+$ is the set of positive roots of $\mathfrak g$;
	\item For $w \in \mathfrak W$, $\ell (w)$ is the length of $w$. That is, $\ell (w)$ is the minimum length of the decomposition of $w$ into a product of simple reflexions $w = s_{\alpha_1} \cdots s_{\alpha_{\ell}}$ where all $\alpha_i$ are simple roots of $\mathfrak g$;
	\item $\rho$ is the half sum of all the positive roots
	\begin{equation} \label{eq:rho}
		\rho = \frac{1}{2} \sum_{\alpha \in R_+} \alpha .
	\end{equation}
\end{itemize}
Throughout, we consider an irreducible representation $W=V_\mu$ of highest weight $\mu$. 

Our proof is inspired by the proof of Weyl dimension formula (following \cite[\S 8.5]{kirillov}). The Weyl dimension formula is a closed formula computing $\dim V_{\mu}$, which can be derived from the Weyl character formula. 
\begin{theorem} [Weyl character formula] \label{thm:weyl_character_formula}
	The formal character of the irreducible $\mathfrak g$-module $V_{\mu}$ of highest weight $\mu$ can be computed from a formal power series expansion of the rational function
	\[ \ch (V_{\mu}) = \frac{\sum_{w \in \mathfrak W} (-1)^{\ell (w)} e^{w . (\mu + \rho)}}{\prod_{\alpha \in R_+} (e^{\alpha / 2} - e^{- \alpha / 2})} .\]
\end{theorem}

Due to its importance to our proof of \Cref{thm:generalized_weyl}, let us review first the proof of Weyl dimension formula. To start, recall the $\mathfrak g$-module $V_{\mu}$ has a weight decomposition $V_{\mu} = \bigoplus_{\theta \in \Lambda} V_{\mu} (\theta)$, where $\Lambda$ is the weight lattice of $\mathfrak g$, $\theta$ runs through the weights of $\mathfrak g$, and $V_{\mu} (\theta)$ indicates the weight $\theta$ subspace of $V_{\mu}$. The formal character of $V_{\mu}$ is an element in the group algebra $\QQ[\Lambda]$ encoding dimensions of the weight subspaces $V_{\mu} (\theta)$:
\[ \ch (V_{\mu}) = \sum_{\theta} \dim V_{\mu} (\theta) \cdot e^{\theta} \quad \in \QQ[\Lambda] .\]

Let us introduce a ring homomorphism ``projection to $\rho$-direction''
\[ \mathrm{pr}_{\rho} : \QQ[\Lambda] \to \QQ[q^{\pm 1}] , \qquad e^{\theta} \mapsto q^{4(\rho, \theta)} ,\]
where $(,)$ is the Killing form of $\mathfrak g$ and $\rho$ is defined in \eqref{eq:rho}. Set $f(q) = f_{\mu} (q) \in \QQ[q^{\pm 1}]$ to be the image of the formal character $\ch (V_{\mu})$ by the homomorphism $\mathrm{pr}_{\rho}$:
\begin{equation} \label{eq:f}
	f(q) = \mathrm{pr}_{\rho} (\ch (V_{\mu})) = \sum_{\theta} \dim V_{\mu}(\theta) \: q^{4(\rho, \theta)} .
\end{equation}
Since $\dim V_{\mu} = \sum_{\theta} \dim V_{\mu} (\theta)$, the dimension of $V_{\mu}$ can be recovered from $f(q)$ by
\begin{equation} \label{eq:dimension}
	\dim V_{\mu} = f(1) .
\end{equation}
On the other hand, if we apply $\mathrm{pr}_{\rho}$ to the Weyl character formula above, then using Weyl denominator identity  (e.g., \cite[Thm 8.39]{kirillov}), the Weyl character formula is translated into
\begin{equation} \label{eq:weyl_character}
	f(q) = \prod_{\alpha \in R_+} \frac{q^{2(\mu + \rho, \alpha)} - q^{-2(\mu + \rho, \alpha)}} {q^{2(\rho, \alpha)} - q^{-2(\rho, \alpha)}} .
\end{equation}
The Weyl dimension formula is obtained by computing $f(1) = \lim_{q \to 1} f(q)$ with the aid of \eqref{eq:weyl_character}.

\medskip

Now let us begin the proof of \Cref{thm:generalized_weyl}. First, notice that for irreducible modules $V_{\mu}$, we can ignore the sign terms $(-1)^k$ in the definition of $s(V_{\mu})$ (i.e., $V_\mu$ is either even or odd). Thus, we have
\[ s(V_{\mu}) = \frac{\sum_k k^2 \dim (V_{\mu})_k}{\dim V_{\mu}} .\]
The following lemma expresses $s(V_{\mu})$ in terms of $f(q)$, imitating \eqref{eq:dimension} above.

\begin{lemma} \label{lem:step1}
	Let $f = f(q)$ be as in \eqref{eq:f}. Then $s(V_{\mu}) = \frac{6}{b(b+1)(b+2)} (\log f)'' (1)$.
\end{lemma}
\begin{proof}
	Consider the derivates of $f$
	\[ f'(q) = \sum_{\theta} 4(\rho, \theta) \dim V_{\mu} (\theta) \: q^{4(\rho, \theta) - 1}, \qquad f'' (q) = \sum_{\theta} 4(\rho, \theta) (4(\rho, \theta) - 1) \dim V_{\mu} (\theta) \: q^{4(\rho, \theta) - 2} .\]
	The Weyl symmetry gives us $\dim V_{\mu} (\theta) = \dim V_{\mu} (-\theta)$. From it, we obtain $f'(1) = 0$ and
	\[ f'' (1) = 16 \sum_{\theta} (\rho, \theta)^2 \dim V_{\mu} (\theta) .\]
	Let us now specialize the discussion to $\mathfrak g = \so (V, q)$ and use the precise value of $\rho$. For special orthogonal Lie algebras, one can compute all the positive roots explicitly in terms of our preferred basis $\varepsilon_i$ and hence obtain the half sum of all the positive roots
	\begin{align} \label{eq:rho_explicit}
		\rho = \begin{cases}
			r \varepsilon_0 + (r-1) \varepsilon_1 + \cdots + \varepsilon_{r-1} & \mbox{when } \ b = 2r \mbox{ is even} \\
			(r + \tfrac{1}{2}) \varepsilon_0 + (r - \tfrac{1}{2}) \varepsilon_1 + \cdots + \tfrac{1}{2} \varepsilon_r \quad & \mbox{when } \ b = 2r + 1 \mbox{ is odd}
		\end{cases} .
	\end{align}
	
	Assume $b = 2r$ is even. Letting $\theta = \sum_{i=0}^r \theta_i \varepsilon_i$, we have $(\rho, \theta) = \sum_{i=0}^r (r-i)\theta_i$. This gives us
	\[ f''(1) = 16 \sum_{\theta} \left( \sum_{i=0}^r (r-i)^2 \theta_i^2 + 2 \sum_{0 \le i < j \le r} (r-i)(r-j) \theta_i \theta_j \right) \dim V_{\mu} (\theta) .\]
	Again by Weyl symmetry, we have $\dim V_{\mu} (\theta) = \dim V_{\mu} (w.\theta)$ for any $w \in \mathfrak W$. Note that the Weyl group $\mathfrak W$ in this case is isomorphic to an order $2$ subgroup of $\mathfrak S_{r+1} \ltimes (\ZZ/2)^{\times (r+1)}$, consisting of the elements with even number of $1 \in \ZZ/2$. The symmetric group part $\mathfrak S_{r+1}$ acts on a weight $\theta = (\theta_0, \cdots, \theta_r)$ by permuting coordinates, and $(\ZZ/2)^{\times (r+1)}$ part acts on it by flipping the signs of $\theta_i$'s. With these symmetries in mind, one deduces
	\begin{itemize}
		\item [(1)] $\sum_{\theta} \theta_i^2 \dim V_{\mu} (\theta) = \sum_{\theta} \theta_j^2 \dim V_{\mu} (\theta)$; and
		\item [(2)] $\sum_{\theta} \theta_i \theta_j \dim V_{\mu} (\theta) = 0$ for $i \neq j$.
	\end{itemize}
	This finally leads us to the identity
	\begin{align*}
		f''(1) &= 16 \cdot \left( r^2 + (r-1)^2 + \cdots + 1^2 \right) \sum_{\theta} \theta_0^2 \dim V_{\mu} (\theta) \\
		&= 16 \cdot \frac{r(r+1)(2r+1)}{6} \sum_k \left( \frac{k}{2} \right)^2 \dim (V_{\mu})_k
		= \frac{b(b+1)(b+2)}{6} \sum_k k^2 \dim (V_{\mu})_k .
	\end{align*}
	Combining it with $f(1) = \dim V_{\mu}$ and $f'(1) = 0$, we have $(\log f)''(1) = \frac{f''(1)}{f(1)} = \frac{b(b+1)(b+2)}{6} s(V_{\mu})$, as claimed.
	
	Next, assume $b = 2r+1$ is odd. Similar argument gives us the computation
	\[ f''(1) = \frac{(r+1)(2r+1)(2r+3)}{3} \sum_k k^2 \dim (V_{\mu})_k = \frac{b(b+1)(b+2)}{6} \sum_k k^2 \dim (V_{\mu})_k .\]
	Hence the same result follows, regardless of the parity of $b$.
\end{proof}

The next step is to use the Weyl character formula \eqref{eq:weyl_character} and compute the value $(\log f)''(1)$.

\begin{lemma} \label{lem:step2}
	Let $f = f(q)$ be as in \eqref{eq:f}. If $\mu_r \ge 0$, then $(\log f)''(1) = \frac{4}{3} b \left[ \left( \sum_{i=0}^r \mu_i \right) b + \left( \sum_{i=0}^r \mu_i^2 - 2i\mu_i \right) \right]$.
\end{lemma}
\begin{proof}
	From the $q$-polynomial version of the Weyl character formula \eqref{eq:weyl_character}, we derive
	\[ \log f(q) = \sum_{\alpha \in R_+} \log \left( \frac{q^{2(\mu + \rho, \alpha)} - q^{-2(\mu + \rho, \alpha)}} {q - 1} \right) - \log \left( \frac{q^{2(\rho, \alpha)} - q^{-2(\rho, \alpha)}} {q - 1} \right) .\]
	Here $q-1$ on the denominators are inserted to make each log terms well-defined in the neighborhood of $q=1$. Notice that $(\log f)''(1)$ is twice the coefficient of the term $(q-1)^2$ in the Taylor series of $\log f$. For a general \emph{positive integer} $a$, the Taylor series expansion of $\log \left( \frac{q^a - q^{-a}}{q-1} \right)$ at $q = 1$ is
	\[ \log \left( \frac{q^a - q^{-a}} {q-1} \right) = \log (2a) - \frac{1}{2} (q-1) + \frac{1}{24} (4a^2 + 5) (q-1)^2 + \cdots ,\]
	which has a degree $2$ coefficient $\frac{1}{24} (4a^2 + 5)$. Since $2(\mu+\rho, \alpha)$ and $2(\rho, \alpha)$ are both positive integers for any $\alpha \in R_+$ (N.B. Here we used the fact $\mu_r \ge 0$), we conclude
	\begin{align*}
		(\log f)''(1) &= 2 \sum_{\alpha \in R_+} \frac{1}{24} \big( 16 (\mu + \rho, \alpha)^2 + 5 \big) - \frac{1}{24} \big( 16 (\rho, \alpha)^2 + 5 \big) \\
		&= \frac{4}{3} \sum_{\alpha \in R_+} (\mu + \rho, \alpha)^2 - (\rho, \alpha)^2 .
	\end{align*}
	
	Recalling \eqref{eq:rho_explicit}, let us get into an explicit computation for $\mathfrak g = \so (V, q)$. Assume $b = 2r$ is even. The positive roots are $R_+ = \{ \varepsilon_i \pm \varepsilon_j : 0 \le i < j \le r \}$. We get
	\begin{align*}
		&\sum_{\alpha \in R_+} (\mu + \rho, \alpha)^2 - (\rho, \alpha)^2 \\
		=& \: \sum_{0 \le i < j \le r} (\mu_i - \mu_j)^2 + 2 (\mu_i - \mu_j) (j-i) + (\mu_i + \mu_j)^2 + 2 (\mu_i + \mu_j) (2r - i - j) \\
		=& \: \sum_{0 \le i < j \le r} \left[ 2(\mu_i^2 + \mu_j^2) + 4r (\mu_i + \mu_j) - (i\mu_i + j\mu_j) \right] \\
		=& \: 2r \left[ 2r \sum_{i=0}^r \mu_i + \sum_{i=0}^r (\mu_i^2 - 2i \mu_i) \right]
		= b \left[ \left( \sum_{i=0}^r \mu_i \right) b + \left( \sum_{i=0}^r \mu_i^2 - 2i \mu_i \right) \right] .
	\end{align*}
	This proves the result in this case.
	
	Similarly for $b = 2r+1$ odd, the positive roots are $R_+ = \{ \varepsilon_i : 0 \le i \le r \} \cup \{ \varepsilon_i \pm \varepsilon_j : 0 \le i < j \le r \}$, giving:
	\begin{align*}
		&\sum_{\alpha \in R_+} (\mu + \rho, \alpha)^2 - (\rho, \alpha)^2 \\
		=& \: \sum_{0 \le i < j \le r} (\mu_i - \mu_j)^2 + 2 (\mu_i - \mu_j) (j-i) + (\mu_i + \mu_j)^2 + 2 (\mu_i + \mu_j) (2r + 1 - i - j) \\
		& \: + \sum_{i=0}^r \mu_i^2 + 2 \mu_i (r + \tfrac{1}{2} - i) \\
		=& \: \sum_{0 \le i < j \le r} \left[ 2(\mu_i^2 + \mu_j^2) + (4r + 2) (\mu_i + \mu_j) - (i\mu_i + j\mu_j) \right] + \sum_{i=0}^r \mu_i^2 + (2r+1) \mu_i - 2i\mu_i \\
		=& \: (2r + 1) \left[ (2r + 1) \sum_{i=0}^r \mu_i + \sum_{i=0}^r (\mu_i^2 - 2i \mu_i) \right]
		= b \left[ \left( \sum_{i=0}^r \mu_i \right) b + \left( \sum_{i=0}^r \mu_i^2 - 2i \mu_i \right) \right] .
	\end{align*}
	This completes the proof of the lemma.
\end{proof}

\begin{proof}[Proof of \Cref{thm:generalized_weyl}]
	Combining \Cref{lem:step1} and \ref{lem:step2}, the theorem follows for the case $\mu_r \ge 0$. Now assume $b_2 = 2r$ is even and $\mu_r < 0$. In this case, $\rho$ does not have the $\varepsilon_r$-coordinate by \eqref{eq:rho_explicit}. Hence, the Weyl character formula (\Cref{thm:weyl_character_formula}) implies that the weights associated to $V_{\mu}$ and $V_{\mu'}$ are bijective via the action $(\theta_0, \cdots, \theta_{r-1}, \theta_r) \mapsto (\theta_0, \cdots, \theta_{r-1}, -\theta_r)$. By definition, the constant $s(W)$ captures only the $h$-eigenspaces, i.e., only the $\varepsilon_0$-coordinates of the weights associated to $W$. This means $s(V_{\mu}) = s(V_{\mu'})$.
\end{proof}

\bibliographystyle{amsalpha}
\bibliography{BettiHK.bib}
\end{document}